\documentclass[11pt]{amsart}

\usepackage[left=1in,right=1in]{geometry}
\usepackage{amssymb,amsmath,amsthm}
\usepackage[flushmargin]{footmisc}
\usepackage[foot]{amsaddr}
\usepackage{wasysym}
\usepackage{graphicx}
\usepackage{epsfig,subfig}
\usepackage{color}

\newcommand{\ode}[2]{\frac{d{#1}}{d{#2}}}
\newcommand{\Tr}{\text{Tr}}

\title{Persistence of saddle behavior in the nonsmooth limit of \\smooth dynamical systems\footnotemark[1]}

\author{Julie Leifeld$^1$
}
\address{$^1$School of Mathematics, University of Minnesota, Minneapolis, MN 55455}
\email{$^1$leif0020@umn.edu}

\author{Kaitlin Hill$^2$
}
\address{$^2$Department of Engineering Sciences and Applied Mathematics, Northwestern University, Evanston, IL 60208}
\email{$^2$k-hill@u.northwestern.edu} 

\author{Andrew Roberts$^3$
}
\address{$^3$Department of Mathematics, Cornell University, Ithaca, NY 14853}
\email{$^3$andrew.roberts@cornell.edu}

\begin{document}

\begin{abstract}
Models such as those involving abrupt changes in the Earth's reflectivity due to ice melt and formation often use nonlinear terms (e.g., hyperbolic tangent) to model the transition between two states.  For various reasons, these models are often approximated by ``simplified'' discontinuous piecewise-linear systems that are obtained by letting the length of the transition region limit to zero.  The smooth versions of these models may exhibit equilibrium solutions that are destroyed as parameter changes transition the models from smooth to nonsmooth.  Using one such model as motivation, we explore the persistence of local behavior around saddle equilibria under these transitions.  We find sufficient conditions under which smooth models with saddle equilibria can become nonsmooth models with ``zombie saddle" behavior, and analogues of stable and unstable manifolds persist.
\end{abstract}


\maketitle

\newcommand{\slugmaster}{\slugger{siads}{201x}{xx}{x}{x--x}}



\pagestyle{myheadings}
\thispagestyle{plain}
\markboth{J. LEIFELD, K. HILL, AND A. ROBERTS}{SADDLE BEHAVIOR IN THE NONSMOOTH LIMIT OF SMOOTH SYSTEMS}

\renewcommand{\thefootnote}{\fnsymbol{footnote}}

{\let\thefootnote\relax\noindent\footnote{The authors would like to acknowledge support of the Mathematics and Climate Research Network (MCRN) under NSF grants DMS-0940363, DMS-0940366, and DMS-0940262.  The work of the second author was supported by a National Science Foundation Graduate Research Fellowship under Grant No. DGE-1324585.}}






\section{Introduction}
\label{sec:introduction}

One of the fundamental concepts in the analysis of smooth dynamical systems is the idea of local behavior.  However, varying a parameter can transform a smooth system into one that is nonsmooth, and nonsmooth systems often have local behavior that is not observed in smooth systems.  Indeed, in many nonsmooth systems, even the standard uniqueness theorem does not apply.   In this paper we analyze persistence of saddle-like behavior in nonsmooth systems.  We conclude that if a saddle equilibrium disappears as a result of the system becoming nonsmooth, and not in the case of a standard smooth bifurcation, saddle behavior will in some sense remain in the nonsmooth system.  We show that the saddle equilibrium in the smooth system limits to a nonsmooth pseudoequilibrium, which is an equilibrium of a Filippov flow.  This pseudoequilibrium is accompanied by stable and unstable manifolds that are analogous to their smooth counterparts.  We call the new pseudoequilibrium point a {\it zombie saddle}, and say the system has zombie saddle behavior.

The system given by
\begin{equation*}
	\dot{\mathbf{x}} = \left\{\begin{array}{cl} p(\mathbf{x}) & h(\mathbf{x})>0 \\ q(\mathbf{x}) & h(\mathbf{x})<0 ,\end{array}\right. 
\end{equation*}
where $p(\mathbf{x})$ and $q(\mathbf{x})$ are smooth functions, is an illustrative example of a nonsmooth system. Notice that the dynamics on the manifold where $h(\mathbf{x})=0$ remain undefined.  This manifold where $h(x)=0$ is called the {\it{discontinuity boundary}}, or {\it{splitting manifold}}, and the system is said to be a piecewise-smooth system.  We are particularly interested in relating the dynamics of a smooth system to those of a nonsmooth approximation to the system obtained through some limiting process. 

{ Conceptual climate models have been formulated using nonsmooth systems dating back at least to Stommel's thermohaline circulation model in 1961.  Over the last decade, the use of nonsmoth systems has become increasingly popular, with \cite{abbot11,jorm,dV07,dV10,Eisenman12,Eisenman09, hogg} all discussing conceptual climate models that are either nonsmooth or exhibit nonsmooth behavior in some limit. In particular, a nonsmooth approximation is often used for the basic energy balance equation 
\begin{equation}
	\label{ebm}
	\frac{dT}{dt} = Q (1-\alpha(T) ) - (A + B T).
\end{equation}  
The albedo function $\alpha(T)$ measures the reflectivity of the Earth's surface.  It is generally assumed to transition from a high albedo at low temperatures due to large quantities of ice to a low albedo at high temperatures.  The transition is often modeled using a hyperbolic tangent:
\begin{equation}
	\label{albedo}
	\alpha(T) = \frac{\alpha_i + \alpha_w}{2} - \frac{\alpha_i - \alpha_w}{2} \tanh \left( \frac{T - T_0}{D} \right).
\end{equation}
$D$ measures how rapidly the transition from high to low albedo occurs.  In the limit $D \rightarrow 0$, the transition is instantaneous, and $\alpha$ is discontinuous.  In \eqref{ebm}, the dependence on greenhouse gas concentration is built into the parameter $A$, with smaller $A$ indicating more greenhouse gases.  However, especially on the time scales of paleoclimatic phenomena like Snowball Earth, $A$ should not be treated as a constant \cite{hogg,walsh2014}.  Introducing $A$ as a (possibly slow) dynamic variable gives the system
\begin{align}
	\label{ebm2D}
	\begin{aligned}
		\dot{T} &= Q (1 - \alpha (T)) - (A+BT) \\
		\dot{A} &= - mA + nT+C,
	\end{aligned}
\end{align}
where $m,n \geq 0 $ \cite{barry2014}.  For certain parameter values the system has three equilibria (see Figure \ref{fig:ebm}(a)), however as $D \rightarrow 0$ the middle equilibrium disappears, as illustrated in Figure \ref{fig:ebm}(b). In the case when the smooth system is bistable and the middle equilibrium is a saddle, the stable manifold to the saddle forms the boundary of the basin of attraction of the two stable equilibria, corresponding to stable climates.
}


\begin{figure}[t]
	\centering
	\subfloat[]{\includegraphics[scale=1]{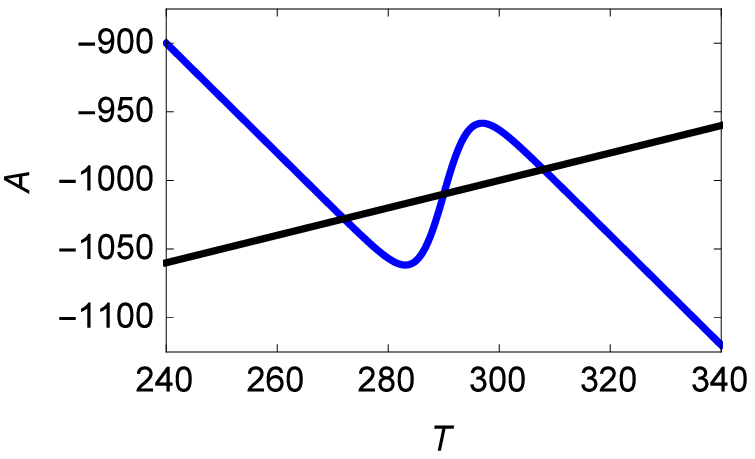}}
	\subfloat[]{\includegraphics[scale=1]{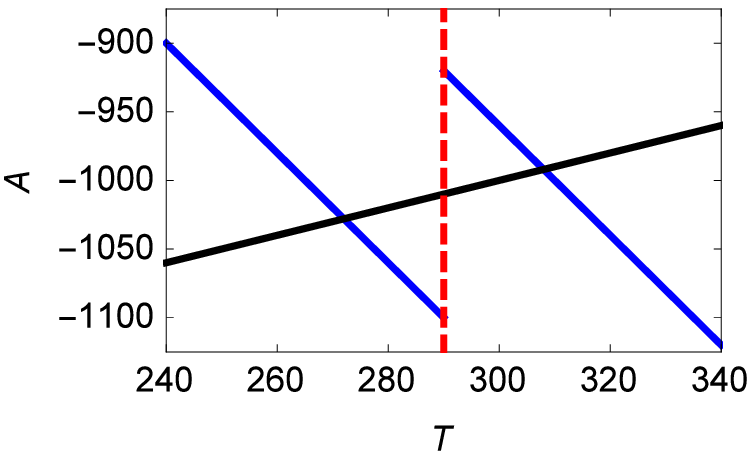}}
	\caption{Nullclines of \eqref{ebm2D} when $m=n=1$, $Q=300$, $\alpha_i=0.8$, $\alpha_w=0.2$, $T_0=290$, $B=4$, $C= -1300$.  (a) shows the case for $D=5$, and (b) shows the limit as $D \rightarrow 0$ with the splitting manifold $T=T_0$ (dashed, red line).  Notice the middle equilibrium from (a) is no longer an equilibrium in (b).}
	\label{fig:ebm}
\end{figure}

In the context of nonsmooth systems, there has been significant interest in the behavior of systems as they transition from smooth to nonsmooth.  For a more in-depth treatment of the recent developments in the dynamics of nonsmooth systems, we refer the reader to the surveys by Colombo, et al. \cite{Colombo12} and Makarenkov, et al. \cite{makarenkov}.  As a system with multiple time scales transitions from smooth to nonsmooth, a technique called pinching has been proposed as a method of preserving dynamics of the smooth system in the nonsmooth limit \cite{simic,Desroches12}. On the other hand, techniques such as blow up and regularization use geometric singular perturbation theory to approximate the behavior in the discontinuity boundary of a nonsmooth system as the system is smoothed\cite{teixeira,Jeffrey14non}.  This process is similar to the blow up method used by Jeffrey \cite{Jeffrey14hidden}, except in \cite{Jeffrey14hidden} more complicated dynamics are allowed to occur in the blown-up splitting manifold.

Nonsmooth systems in which an equilibrium collides with a discontinuity boundary as a parameter varies are said to undergo a border-collision bifurcation \cite{diBernardo,Kuznetsov03,nusse}. However, the interaction between an equilibrium colliding with a discontinuity boundary as the system itself is limiting to a nonsmooth limit is not well-understood. Smooth systems that limit to a nonsmooth variant as some parameter varies have held recent interest, notably in the context of climate dynamics \cite{Welander82,Abshagen04,Eisenman12,widiasih}. In this paper we specifically study systems with saddle equilibrium points that disappear in a traditional sense as the system becomes nonsmooth.  By this we mean that the equilibrium limits to the nonsmooth boundary as $a\rightarrow 0$, and hence is not a place where the vector field defined in the system is $0$. Before completing the analysis, we first define some syntax in Section \ref{sec:techniques}.  In Section \ref{sec:Results} we state and prove the main results of the paper. We consider three cases of systems along with their nonsmooth limits: a system that has no sliding region in the nonsmooth limit, a system with a repelling sliding region, and a system with an attracting sliding region. In all cases we prove the existence of saddle equilibria in the systems throughout the limiting process.  Additionally, we show show that the saddle behavior persists in the limit.  We call these persistent pseudoequilibria zombie saddles, because they are the remnants of ``real" saddle equilibria from the smooth systems. Then in Section \ref{sec:examples} we provide examples of systems which exhibit this type of behavior. We conclude in Section \ref{sec:discussion} with a discussion of our results.

\section{Nonsmooth Analysis Techniques}
\label{sec:techniques}

Our analysis uses the standard syntax of nonsmooth systems.  For simplicity we will assume there is a unique $C^1$ splitting manifold of codimension 1.  Then, we can define a scalar function $h(\mathbf{x}):\mathbb{R}^n\rightarrow \mathbb{R}$, with $h(\mathbf{x})=0$ along the manifold.  In the regions where the system is smooth, we let $f^+(\mathbf{x})$ be the vector field where $h(\mathbf{x})>0$, and define $f^-(\mathbf{x})$ similarly.  Let $\lambda$ be a step function, defined 
\[
\lambda=\left\{\begin{array}{cl} 1 & h(\mathbf{x})>0\\ 0 & h(\mathbf{x})<0. \end{array}\right.
\]
Then, as in \cite{filippov}, the entire system can be written in the form of a convex combination of the vector fields
\begin{align}
\label{fc}
\dot{\mathbf{x}}= \mathbf{f}(\mathbf{x},\lambda)&=\left\{\begin{array}{cc} f^+(\mathbf{x}) & h(\mathbf{x})>0 \\ f^-(\mathbf{x}) & h(\mathbf{x})<0 \end{array}\right. \nonumber \\
 &= \lambda f^+(\mathbf{x})+(1-\lambda)f^-(\mathbf{x}).
\end{align}
A sliding solution exists along the nonsmooth boundary if the following system can be solved for $\lambda\in[0,1]$.
\begin{equation}
\label{fsol}
\begin{array}{r}S=\mathbf{f}(\mathbf{x},\lambda)\cdot\nabla h(\mathbf{x})=0\\
h(\mathbf{x})=0
\end{array}
\end{equation}
In a Filippov system, solutions to \eqref{fsol} are only possible in regions where the dot product $f^+\cdot f^-<0$.  Intuitively, one would expect sliding solutions here, as solutions should cross the splitting manifold when the vector fields point in the same direction.  Therefore, sliding regions are defined as regions along the splitting manifold where $f^+\cdot f^-<0$.  The boundaries of the sliding region are then places where either vector field is tangent to the splitting manifold, or more formally, when $f^{\pm}\cdot\nabla h =0$.  A tangency in $f^+$ is called visible if a solution through the tangent point exists entirely in $f^+$, and is invisible if the solution exists entirely in $f^-$.  Corresponding definitions exist for tangencies in $f^-$.  The stability of the sliding solution is determined by the sign of $$\ode{}{\lambda}S$$ along $h(x)=0$, with $$\ode{}{\lambda}S<0$$ indicating a stable sliding region.  Alternatively, $$\ode{}{\lambda}S>0$$ indicates that the sliding region is unstable.  See Figure ~\ref{fig:generic} for an illustration of attracting and repelling sliding regions. Recent work by Jeffrey has expanded on this method, by using the idea that convex combinations of $f^+$ and $f^-$ are not necessary to the analysis \cite{Jeffrey14hidden}.  For example, adding nonlinear terms of the parameter $\lambda$, such as $\lambda(1-\lambda)$, which are zero away from the nonsmooth boundary, also gives a method of finding sliding solutions which are beyond the scope of the standard Filippov analysis.  These solutions can exist in regions for which the vector field points in the same direction, and hence one would not expect sliding solutions.  To avoid these difficulties, in this paper we assume the nonsmooth functions can be written linearly with respect to $\lambda$.  In other words, we let $g_a(x)$ be a smooth, monotonically increasing function, and as $a$ approaches 0, $g_a(x)$ approaches $g_0(x)$, with 
\begin{equation}
\label{gdef}
g_0(x)=\left\{\begin{array}{cl}g_+(x) & x>0\\ g_-(x) & x<0.\end{array}\right. 
\end{equation}
It is possible to define $g_0(0)$ based on a pointwise convergence of $g_a(0)$.  However, this precludes the possibility of any dynamic behavior along the splitting manifold which may be indicated by real behavior in the application system.  Therefore, we instead use the nonsmooth analysis conventions previously described, and also assume
\[
g_0(x)=\lambda g_+(x)+(1-\lambda)g_-(x).
\]
It is important to note the breakdown of uniqueness of solutions in sliding regions of nonsmooth systems.  Several initial conditions can generate solutions that approach the nonsmooth boundary in either forward or backward time and intersect. In Figure \ref{fig:generic}(a), solutions are nonunique in backward time. On the other hand, in Figure \ref{fig:generic}(b) solutions are nonunique in forward time; that is, only one solution trajectory can enter the sliding region from either the right-hand or left-hand side, but a family of solutions leaves the sliding region from both sides. Hence, uniqueness in both forward and backward time is not to be expected in these systems.  However, as we will see in Section \ref{sec:Results}, it may still be possible to choose unique trajectories which correspond to stable and unstable manifolds.  

\begin{figure}[!hb]
	\centering
	\subfloat[]{\includegraphics[scale=1]{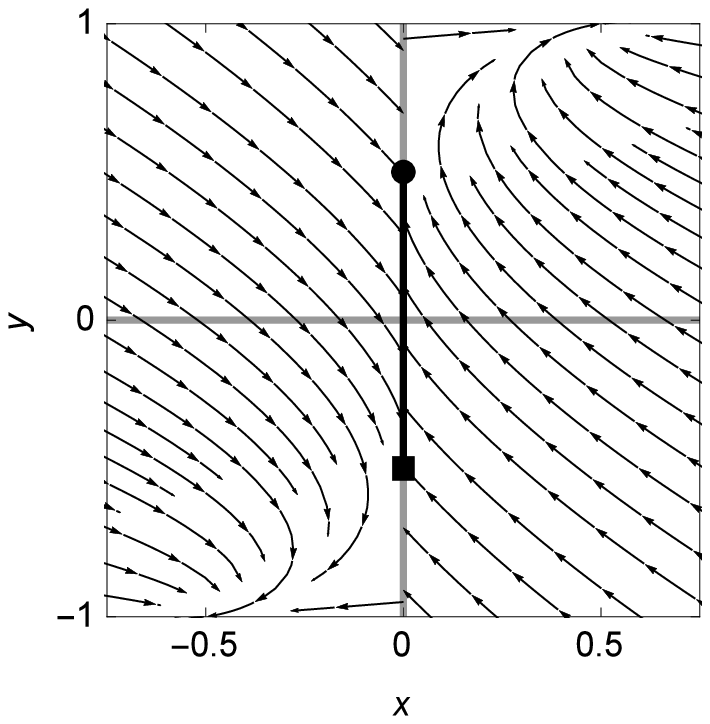}}
	\subfloat[]{\includegraphics[scale=1]{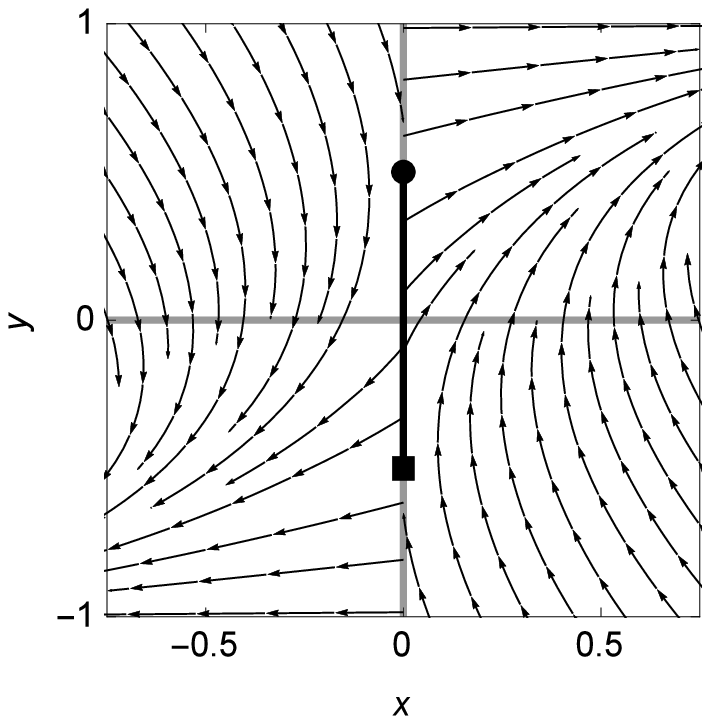}}
\caption{Typical phase diagrams for nonsmooth systems with sliding regions, with (a) attracting sliding and (b) repelling sliding. In (a), trajectories may enter the sliding region at any point along the line between the dot and the square, while in (b) trajectories may only enter the sliding region at the dot from the left and at the square from the right.}
\label{fig:generic}
\end{figure}

\section{Results}
\label{sec:Results}

In this section we prove the persistence of saddle equilibria in three distinct cases, when the sliding region is stable, unstable, or nonexistent.  In each case, we prove the existence of saddle equilibria in the smooth systems, as well as saddle psuedoequilibria in the corresponding nonsmooth systems.  We also prove that stable and unstable manifolds persist in some form.

\subsection{No sliding solutions}
\label{sec:resultsnoslide}

$ $\newline\vspace{-1em}
\newtheorem{thmnoslide}{Theorem}[section]
\begin{thmnoslide}
Let $$g_a(x)=g\left(\frac{x}{a}\right),$$ $a\neq0$, be a monotonically increasing, smooth function such that $g_a(x)\rightarrow g_0(x)$ as $a\rightarrow 0$, with 
\[
g_0(x)=\left\{\begin{array}{cl}g_+(x) & x>0 \\ g_-(x) & x<0, \end{array}\right. 
\]
where $g_+$ and $g_-$ are smooth functions up to the boundary $x=0$.
Furthermore, assume $g_0 = \lambda g_+(x)+(1-\lambda)g_-(x)$, where 
\[
\lambda = \left\{\begin{array}{cl} 1& x>0 \\ 0 & x<0.\end{array}\right.
\]
Finally, let $g_-(0)<f(0)<g_+(0)$, where $f$ is a smooth function.  Then, when $0<a\ll1$, the system 
\begin{equation}
\label{dblTangency}
\begin{array}{rcl}
\dot{x}&=&y-f(x)\\
\dot{y}&=&g_a(x)-y,
\end{array}
\end{equation}
contains a saddle equilibrium point $p_a=(x_a,y_a)$, such that $x_a\rightarrow0$ as $a\rightarrow0$.  When $a=0$, the nonsmooth system contains a pseudoequilibrium on the splitting manifold {$x=0$}, and stable and unstable manifolds exist locally.
\end{thmnoslide}

\begin{proof}

First we look at the case $a\neq0$.  Equilibrium solutions are found by solving $f(x)-g_a(x)=0$.  Consider an arbitrarily small interval around $x=0$, $(-\delta,\delta)$.  Let $d=f(0)-g_+(0)<0$.  Choose $x_1\in (0,\delta)$ such that 
\[
\begin{array}{rcl}
|f(x_1)-f(0)|&<&\dfrac{|d|}{3}\\[10pt]
|g_+(0)-g_+(x_1)|&<&\dfrac{|d|}{3}
\end{array}
\]
It is possible to choose this $x_1$ because of the continuity of $f$ and $g_+$.  Then, choose $a=a_1\neq0$ such that $$|g_+(x_1)-g_a(x_1)|<\frac{|d|}{3}.$$  This small $a$ exists because of the pointwise convergence of $g_a$.  We show that $f(x_1)-g_a(x_1)<0$.
\begin{align*}
f(x_1)-g_a(x_1)&=f(0)-g_+(0)+\left(g_+(0)-g_+(x_1)\right)+\left(f(x_1)-f(0)\right)+\left(g_+(x_1)-g_a(x_1)\right) \\
&<d+|d| \\
&=0
\end{align*}
A similar argument works for $x_2\in(-\delta,0)$.  Therefore we conclude by the Intermediate Value Theorem that $f(x)-g_a(x)=0$ for some $x\in(x_2,x_1)\subset(-\delta,\delta)$.

To see that this equilibrium is a saddle, it is convenient to make a change of variables, which is valid for all $a\neq 0$.  Let $\xi=x/a$.  Then, the system becomes
\[
\begin{array}{rcl}\dot{\xi}&=&\dfrac{1}{a}\left(y-f(a\xi)\right)\\
\dot{y}&=&g(\xi)-y.
\end{array}
\]
The Jacobian of this system is
\[
J=\begin{bmatrix}-f'(a\xi) & \frac{1}{a}\\ g'(\xi)&-1\end{bmatrix}
\]
with $$\det J = f'(a\xi)-\frac{1}{a}g'(\xi).$$  Because $g$ is monotone increasing, we know $g'(\xi)>0$.  Thus, it is immediately clear that for sufficiently small $a$, $\det J <0$.

\vspace{1em}

When $a=0$, the system is discontinuous and the splitting manifold is the line $x=0$.  The boundaries of the Filippov sliding region occur where $y=f(0)$.  Given that $f$ is a function, there is only one point at which this occurs, meaning that there is no linear sliding region.  This point is a pseudoequilibrium of the Filippov flow, as the vector fields on either side of the splitting manifold are colinear and opposite here.  It is easy to check that the the tangency is visibile on each side, as seen in Figure \ref{fig:no-sliding}(a).

Furthermore, we can show the existence of trajectories through this tangency.  First, we consider the system defined where $x<0$, and the case $f'(0)>0$.  Let $y>f(0)$.  Then, because $f(0)>g_-(0)$, there is a region $ x\in(-\delta,0)$, $y\in(f(0),f(0)+\varepsilon)$ for which $\dot{y}<0$.  Moreover, because $f'(x)>0$, $\dot{x}>0$ in this region.  This implies that there are no strictly invariant sets in the region, and therefore every point can be taken continuously to the left and upper edges of the boundary using the map defined by the flow. In particular, there is a trajectory in this set which maps to $(0,f(0))$, the tangency point, under the flow.  This trajectory is an analogue of a stable manifold, along which the equilibrium can be reached in finite time.  Because $(0,f(0))$ is a tangency of the vector field when $x<0$, the trajectory remains in the $x<0$ region, and the portion which approaches $(0,f(0))$ in backward time likewise acts as an analogue of an unstable manifold.  If $f'(0)<0$, we can make a similar argument after reversing time.  Then, there is a region $x\in(-\delta,0)$, $y\in(f(0),f(0)-\varepsilon)$ for which $\dot{x}>0$ and $\dot{y}>0$.  This implies the existence of in initial condition in this region which approaches the tangency $(0,f(0))$ in backward time, with respect to the original system.  This trajectory will remain in the region $x<0$ for the same reason, and acts as an analogue to stable and unstable manifolds.  The same arguments can be made for the existence of a trajectory through the tangency in the region $x>0$.  Thus in the full system, trajectories through the double tangency are not unique.  Much like a saddle, there are two trajectories approaching the point, and two trajectories leaving.  These trajectories function as stable and unstable manifolds to a saddle point, especially in their role as separatrices.  Much like stable and unstable manifolds, their intersection consists of one point, however, in this case that point is reachable in finite time. See Figure \ref{fig:no-sliding}(b) for an illustration of this idea, where the dark green boxes represent the boundaries of the flow in the regions discussed. 
\end{proof}

\begin{figure}[!ht]
	\centering
	\subfloat[]{\includegraphics[scale=1]{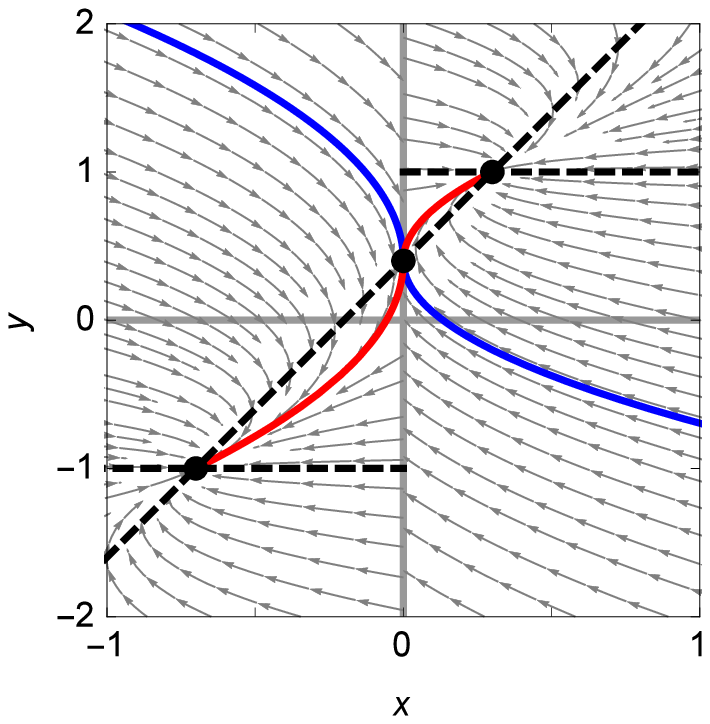}}
	\subfloat[]{\includegraphics[scale=1]{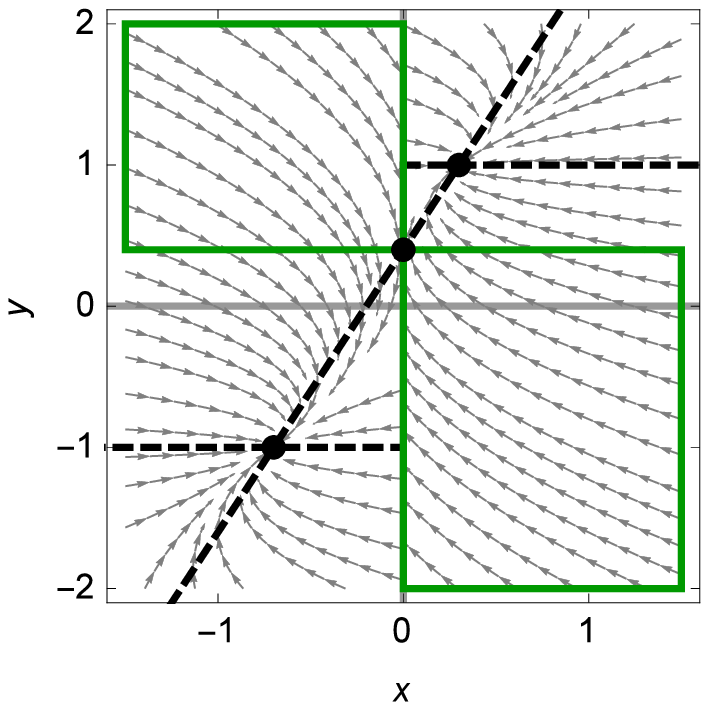}}
\caption{(a) Example phase diagram of the nonsmooth limit of (\ref{dblTangency}), with $f(x)=2x+0.4$ and $g_a(x)=\tanh(\frac{x}{a})$, so that $g_0(x)=\pm 1$. Blue curves represent the trajectories that have behavior similar to a stable manifold, and red curves exhibit behavior similar to an unstable manifold. The black points show the equilibria of the system. Dashed lines represent the nullclines of the system. (b) An illustration of the visible tangency using $f$ and $g$ as defined in (a). The dashed lines show the nullclines of the system, and the solid green lines give an illustration of the boxes considered in finding the tangency.}
\label{fig:no-sliding}
\end{figure}

\subsection{A repelling sliding region}

$ $\newline\vspace{-1em}
\newtheorem{thmrepslide}{Theorem}[section]
\begin{thmrepslide}
Under the assumptions in Theorem 1, when $0<a\ll1$, the system 
\begin{equation}
\label{rSliding}
\begin{array}{rcl}
\dot{x}&=&g_a(x)-y\\
\dot{y}&=&f(x)-y,
\end{array}
\end{equation}
contains a saddle equilibrium point $x_a$, such that $x_a\rightarrow0$ as $a\rightarrow0$.  When $a=0$, the nonsmooth system contains a stable pseudoequilibrium in a repelling sliding flow on the splitting manifold $x=0$, and stable and unstable manifolds exist locally.

\end{thmrepslide}

\begin{proof}

Again, equilibria are determined by the equation $f(x)-g_a(x)=0$.  An identical calculation to that in the proof of Theorem 1 shows that an equilibrium solution exists in $(-\delta,\delta)$.  The calculation will not be repeated here.  
Doing the same convenient coordinate change as before, we get a new system
\[
\begin{array}{rcl}
\dot{\xi}&=&\frac{1}{a}\left(g(\xi)-y\right)\\
\dot{y}&=&f(a\xi)-y.
\end{array}
\]
In this system the Jacobian is 
\[
J=\begin{bmatrix} \frac{1}{a}g'(\xi) & -\frac{1}{a}\\ a f'(a\xi) &-1 \end{bmatrix}
\]
with $$\det J= f'(a\xi)-\frac{1}{a}g'(\xi).$$  Again we have that $\det J<0$, giving us a saddle equilibrium.

\vspace{1em}

The nonsmooth version of this system also shows saddle like behavior as $g_a(x)$ approaches a step function in the $a=0$ limit. See Figure \ref{fig:repelling-sliding} for an example of a phase diagram of (\ref{rSliding}). In this case, tangencies to the discontinuity boundary occur when $y=g_+(0)$ and $y=g_-(0)$, implying that the sliding region has nonzero length.  Moreover, this sliding region is repelling, which means that any solution starting on the sliding region will leave arbitrarily.  To see this, we will use some formalism.  Recall that $g_0(x)=\lambda g_+(x)+(1-\lambda)g_-(x)$.  We define $h(x,y)$ to be a scalar function such that $h>0$ when $x>0$, and $h<0$ when $x<0$.  Then, sliding solutions exist where $\mathbf{f}(\mathbf{x},\lambda)\cdot \nabla h =0$ can be solved for some $\lambda\in[0,1]$ along $x=0$.  In this case, it is clear that sliding solutions exist on $[g_-(0),g_+(0)]$, as asserted previously.  Moreover, as in \cite{Jeffrey14hidden}, stability of the sliding region can be determined by 
\[
S(\mathbf{x},\lambda)=\ode{}{\lambda}\mathbf{f}\cdot\nabla h,
\]
where $S<0$ implies an attracting sliding region, and $S>0$ implies repelling.  In this example, $h=x$, so \[
\mathbf{f}\cdot \nabla h =\lambda (g_+(0)-y)+(1-\lambda)(g_-(0)-y).
\]
\[\ode{}{\lambda}\mathbf{f}\cdot\nabla h = g_+(0)-g_-(0)>0,
\]
indicating repelling sliding.
  
\vspace{1em}

We can also use this formulation to find the flow along the sliding region.  Using \eqref{fsol} to solve for $\lambda$ and plugging this solution into the original equations, we get an equation for the flow,
\[
\dot{y}=f(0)-y.
\]
It is clear that this has one stable equilibrium solution, $y=f(0)$.

\vspace{1em}

Because the repelling sliding region has one stable equilibrium, the system has behavior analogous to a saddle.  This saddle does not have unique trajectories acting as separatrices, as Section \ref{sec:resultsnoslide}.  Instead, a family of trajectories will move along the sliding region, and each will leave the manifold at some point before it can reach the equilibrium.  However, the sliding region itself might be considered analogous to a stable manifold.  For the same reasons as described in section \ref{sec:resultsnoslide}, there is a unique trajectory on each side of the discontinuity which emanates from the equilibrium on the sliding region.   These two trajectories are analogous to an unstable manifold.
\end{proof}

\vspace{1em}

\begin{figure}[!ht]
	\centering
	\includegraphics[scale=1]{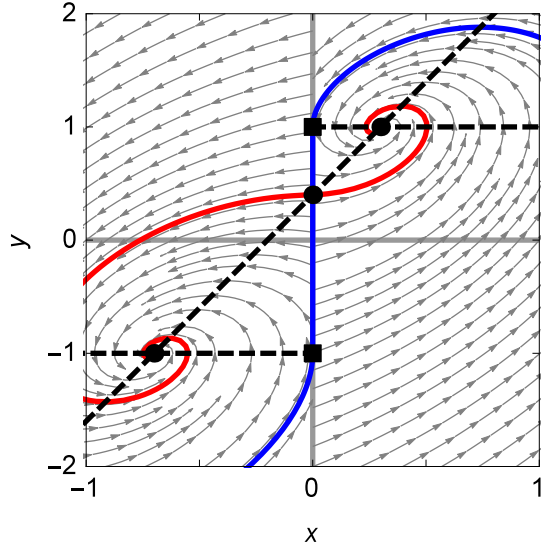}
\caption{Example phase diagram of the nonsmooth limit of (\ref{rSliding}), with $f(x)=2x+0.4$ and $g_a(x)=\tanh(\frac{x}{a})$, so that $g_0(x)=\pm 1$. Blue curves represent the trajectories that have behavior similar to a stable manifold, and red curves correspond to behavior similar to an unstable manifold. The points show the equilibria of the system, and the squares show where the sliding region begins/ends. Dashed black lines represent the nullclines of the system.}
\label{fig:repelling-sliding}
\end{figure}

\subsection{An attracting sliding region}

$ $\newline\vspace{-1em}
\newtheorem{thmattslide}{Theorem}[section]
\begin{thmattslide}
Under the assumptions in Theorem 1, when $0<a\ll1$, the system 
\begin{equation}
\label{aSliding}
\begin{array}{rcl}
\dot{x}&=&y-g_a(x)\\
\dot{y}&=&y-f(x),
\end{array}
\end{equation}
contains a saddle equilibrium point $x_a$, with $x_a\rightarrow0$ as $a\rightarrow0$.  When $a=0$, the nonsmooth system contains an unstable pseudoequilibrium in a stable sliding region on the splitting manifold $x=0$, and stable and unstable manifolds exist locally.

\end{thmattslide}

\begin{proof}
The system in Theorem 3 is obtained by reversing time in the previous system, so we need not repeat the proof that a saddle equilibrium exists and limits to x=0. 

\vspace{1em}

The nonsmooth system again has two tangencies along the discontinuity line.  These occur at $y=g_+(0)$ and $y=g_-(0)$.  The sliding region in this case is attracting, because $S(x,\lambda)=g_-(0)-g_+(0)<0$.  In this system, the sliding flow is given by the equation
\[
\dot{y}=y-f(0)
\]
from which it is again immediately clear that the sliding flow has one unstable equilibrium.
The equilibrium point is unstable, but has no unique unstable manifold.  Instead, a family of trajectories will be attracted to the stable sliding region, and then slide away from the unstable equilibrium.  The sliding region itself might again be considered as the analogue of the unstable manifold.  Again, there will be unique solutions to the equations on either side of the discontinuity which hit the sliding region exactly at the equilibrium of the sliding flow. See Figure \ref{fig:attracting-sliding}. These trajectories arrive at the equilibrium in finite time, and act as analogues to the stable manifold.  So, again, the nonsmooth system displays saddle behavior, albeit slightly different from the saddles of the first two nonsmooth systems.
\end{proof}

\vspace{1em}

\begin{figure}[!ht]
	\centering
	\includegraphics[scale=1]{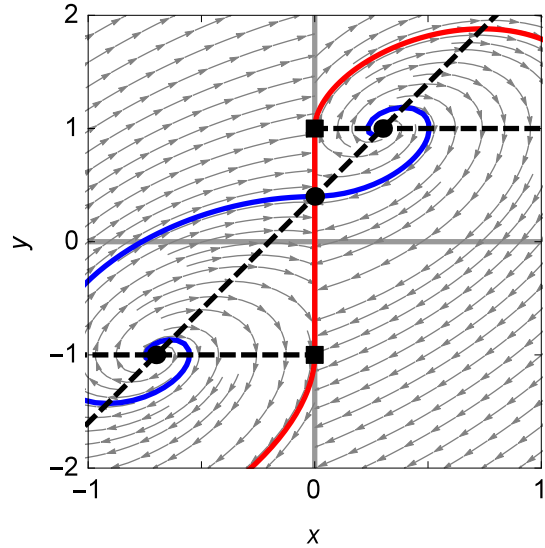}
\caption{Example phase diagram of the nonsmooth limit of (\ref{aSliding}), with $f(x)=2x+0.4$ and $g_a(x)=\tanh(\frac{x}{a})$, so that $g_0(x)=\pm 1$.  Red curves represent the trajectories that have behavior similar to an unstable manifold, and blue curves correspond to a stable manifold. The points show the equilibria of the system, and the squares show where the sliding region begins/ends. Dashed black lines represent the nullclines of the system.}
\label{fig:attracting-sliding}
\end{figure}

\section{Examples}
\label{sec:examples}

\subsection{A nonsmooth system that exhibits all three saddles}
Finally, we find a nonsmooth system which displays all three types of saddle behavior.  In this example we let $g_a(x)=\tanh\left(\frac{x}{a}\right)$, which limits to 
\[
g_0(x) = \left\{\begin{array}{cl}1 & x>0 \\ -1 & x<0. \end{array}\right.
\]
The smooth system is as follows, where $\alpha>0$. 
\[
\begin{array}{rcl}
\dot{x}&=&kg_a(x)-x+y\\
\dot{y}&=&\alpha g_a(x)-y.
\end{array}
\]
\[
J=\begin{bmatrix} -1+kg'_a(x_a) &1\\ \alpha g'_a(x_a) & -1 \end{bmatrix},
\]
and $\det J=1-(k+\alpha)g'_a(x_a)$.  This means that for $\alpha>0$ fixed and $|k|<\alpha$, the conditions given previously for $g_a(x)$ again guarantee that the equilibrium $x_a$ is a saddle point. 

The nonsmooth limit of the previous system is
\begin{equation}\label{3-saddle}
\begin{array}{rcl}
\dot{x}&=&kg_0(x)-x+y\\
\dot{y}&=&\alpha g_0(x) -y,
\end{array}
\end{equation}
where again
\[
g_0(x)=\left\{\begin{array}{cl} 1 & x>0 \\ -1 & x<0. \end{array}\right. = 2\lambda-1,
\]
if $\lambda$ is defined as before.  Tangencies to the discontinuity occur at $y=-k(2\lambda-1)$. See Figure \ref{fig:3-saddle} for illustrations of the phase diagram for this system for varied $k$. For this system, $\mathbf{f}\cdot\nabla h = k(2\lambda-1)-x-y$.  So, $\ode{}{\lambda}S(\mathbf{x},\lambda)=2k$, and when $k>0$, the sliding region is repelling.  When $k\neq0$, the sliding flow in this system is given by
\[
\dot{y}=-\left(\frac{\alpha}{k}+1\right)y.
\]
The equilibrium occurs at $y=0$, and is stable if $\alpha>-k$, which is true for any positive $k$.  Similarly, when $k<0$, the sliding region is attracting.  In this case, the equilibrium is stable if $\alpha>|k|$, so in some region where $|k|$ is small, this is satisfied.  

\vspace{1em}

Because of the division by $k$ in the sliding flow, the case where $k=0$ must be treated separately.  In this case the system is 
\[
\begin{array}{rcl}\dot{x}&=&-x+y\\
\dot{y}&=&\alpha(2\lambda-1) -y.
\end{array}
\]
Here, tangencies occur when $y=x=0$, so there is no sliding region.  Instead, as in Section 2.1, we have saddle behavior through two visible tangencies occurring at the same point, $y=0$.

\begin{figure}[!ht]
	\centering
		\subfloat[]{\includegraphics[scale=1]{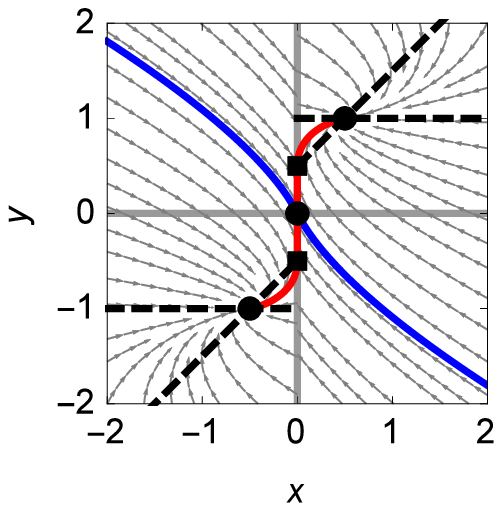}}
		\subfloat[]{\includegraphics[scale=1]{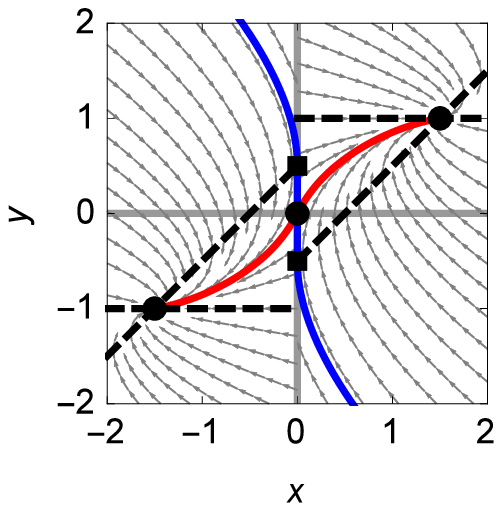}}
		\subfloat[]{\includegraphics[scale=1]{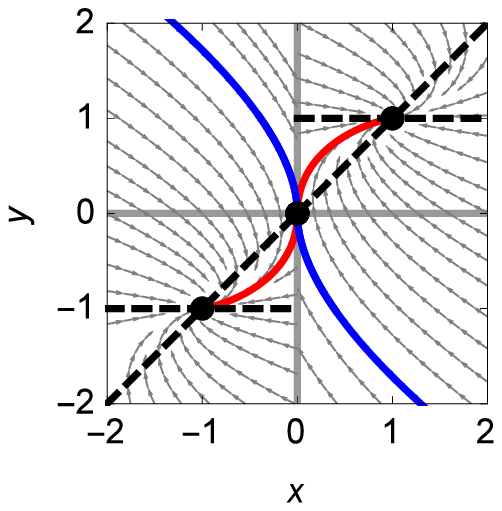}}
\caption{Example phase diagrams of (\ref{3-saddle}), with $\alpha=1$ and (a) $k<0$, (b) $k=0$, and (c) $k>0$. As before, red curves represent the trajectories that have behavior similar to an unstable manifold, and blue curves exhibit behavior similar to a stable manifold. The points show the equilibria of the system, and the squares show where the sliding region begins/ends. Dashed black lines represent the nullclines of the systems.}
\label{fig:3-saddle}
\end{figure}

\subsection{Energy balance model}
{
In this section we analyze the energy balance model \eqref{ebm} from the introduction.  We will reproduce the equations for reference:
\begin{equation*}
	\begin{array}{rl}
		\dot{T} &= Q (1 - \alpha (T)) - (A+BT) \\
		\dot{A} &= - mA + nT+C,
	\end{array}
\end{equation*}
where 
$$
\alpha(T) = \frac{\alpha_i + \alpha_w}{2} - \frac{\alpha_i - \alpha_w}{2} \tanh \left( \frac{T - T_0}{D} \right).
$$
As we will see, the system has a repelling sliding region as $D \rightarrow 0.$  The system is not quite in the same form as the system in Theorem 2 as is, however there is a change of variables by which it is possible.  The resulting system is no more enlightening than \eqref{ebm} is in its current form, so we forgo the change of variables and analyze the system as is.  Computing the Jacobian gives
\begin{equation}
J = \left( \begin{array}{cc}
-Q \alpha'(T)-B 	& 	-1 \\
	n 		& 	-m \end{array} \right),
\end{equation}
and $\det (J) = m (Q \alpha'(T) +B) +n$.  The dependence of $\alpha$ on $T$ is through a hyperbolic tangent with a negative coefficient $(\alpha_i -\alpha_w)/2$.  In the limit as $D \rightarrow 0,$ $\alpha(T)$ limits to 
\begin{equation}
\alpha_0 (T) = \left\{ \begin{array}{ll}
\alpha_i & \text{if } T < T_0 \\
\alpha_w & \text{if } T> T_0.
\end{array} \right.
\end{equation}
Define $A_i = Q(1-\alpha_i) - B T_0$ and define $A_w$ similarly.  If the $A$-nullcline intersects the line $T=T_0$ between $A_i$ and $A_w$, then Figure \ref{fig:ebm} indicates there will be an equilibrium near $T_0$ at say, $T_*(D),$ that limits to $T_0$ as $D \rightarrow 0$.  This occurs if 
$$A_i <  \frac{n}{m} T_0 + C < A_w.$$
Also, in the $D=0$ limit, the derivative of $\alpha(T_*(D)) \rightarrow -\infty.$  So for $D$ small enough, the equilibrium at $T_*(D)$ will be a saddle.  As shown in Figure \ref{fig:climate}, \eqref{ebm} limits to a system with a repelling sliding region with an attracting critical point.  The two proper equilibria that remain in the nonsmooth system are fully attracting since at these points $\det(J) = mB+n >0$ and $\Tr(J) =-(B+m)<0.$  Thus, the zombie saddle at $T_0$ serves as an unstable object between the two attracting equilibria.  

}
\begin{figure}[!ht]
	\centering
	\includegraphics[scale=1]{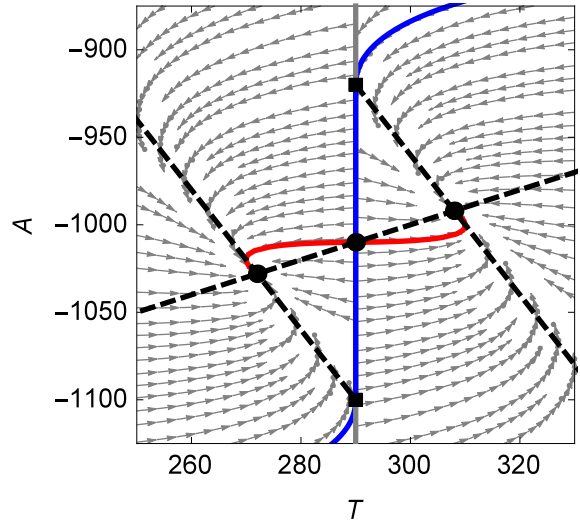}
\caption{Phase diagram of (\ref{ebm2D}) in the nonsmooth limit. Parameter values are the same as in Figure \ref{fig:ebm}, i.e. $m = n = 1,\, Q = 300,\, i = 0.8,\, w = 0.2,\, T_0 = 290,\, B = 4$, and $C = 1300$.  Blue curves represent stable manifold-like behavior, red curves correspond to unstable manifold-like behavior, and the dashed lines represent the nullclines of the system. The points show the equilibria of the system, and the squares show where the sliding region begins/ends.}
\label{fig:climate}
\end{figure}

\section{Discussion}
\label{sec:discussion}
It has been established that there are perils associated with using a nonsmooth approximation of a smooth dynamical system \cite{Jeffrey11}, and one might view this paper as yet another chapter of that cautionary tale---after all, a nonsmooth system may not even reflect the number of equilibrium points of the system it is supposed to approximate.  On the other hand, we view this as highlighting similarities between certain smooth and nonsmooth systems.  In this case, there is saddle behavior in the nonsmooth system that mimics what is observed in the smooth system.  The saddle behavior is caused either by a double-tangency along the splitting line as in \eqref{dblTangency} or by a an equilibrium of a sliding solution as in systems \eqref{rSliding} and \eqref{aSliding}.  We refer to the limiting points in the nonsmooth systems as having ``zombie behavior'' to reflect that even though the saddle equilibrium was destroyed, the saddle behavior still exists in an altered state.  

The recognition of the saddle behavior, especially the existence of trajectories analogous to stable manifolds of a zombie saddle, can be a useful tool for analyzing a dynamical system.  { Under certain parameters, the motivating conceptual climate model \eqref{ebm} exhibits bistability.}  In the smooth system, there is a saddle point between the two attracting equilibria, and the stable manifold of the saddle separates the basins of attraction of the stable equilibria.  Analysis of this sort is standard in smooth systems.  The existence of a zombie saddle allows us to carry out similar analysis in the corresponding nonsmooth system.  Future exploration may examine the manner in which (un)stable manifold of a saddle in the smooth system limits to the (un)stable manifold of the zombie saddle.  The computational difficulty in this endeavor arises because the zombie saddle is not truly an equilibrium, and thus dynamics near the point are not described by its linearization.    Therefore, we cannot rely on the eigenvalues and eigenvectors of the linear system to provide the desired information.  

In smooth systems, if the stable manifold of a saddle is co-dimension one, then it forms a separatrix, dividing phase space into regions with different dynamics.  For this reason, it would be useful to find similar results for zombie saddles in higher dimensions where computing basins of attraction is generally more difficult.  

\section*{Acknowledgments}
We would like to thank Mike Jeffrey, Richard McGehee, and Mary Silber for their conversations throughout the process.

\bibliographystyle{siam}
\bibliography{nonsmooth}

\end{document}